\documentclass[11pt]{article}
\usepackage{amssymb}
\usepackage[all,arc]{xy}
\usepackage{mathrsfs}
\usepackage{extarrows}
\usepackage{xcolor}
\usepackage{amsmath}
\usepackage{amsfonts,epsfig}
\usepackage{amssymb,bm}
\usepackage{graphicx}
\usepackage[T1]{fontenc}
\usepackage{subfig}
\usepackage{pgf}
\usepackage{appendix}

\usepackage{diagbox}
\usepackage{tabularx}  
\usepackage{booktabs}
\usepackage{multirow}
\usepackage{array,makecell}

\usepackage{algorithm}
\usepackage{algpseudocode}
\usepackage{pifont}

\usepackage{tikz}
\usetikzlibrary{decorations.pathreplacing}

\numberwithin{equation}{section}

\usepackage{amsmath,amsthm,amsfonts,amssymb,amscd,color}
\usepackage{hyperref}
\hypersetup{
	colorlinks=true,
	linkcolor=blue,
	anchorcolor=blue,
	citecolor=blue}

\allowdisplaybreaks[4]

\newtheorem{theo}{Theorem}[section]
\newtheorem{lem}{Lemma}[section]
\newtheorem{cor}{Corollary}[section]

\newtheorem{remark}{Remark}[section]

\newtheorem{prob}{Problem}[section]

\newcommand{\F}{\mathbb{F}}
\newcommand{\0}{\mathbf{0}}
\newcommand{\one}{\mathbf{1}}
\newcommand{\vv}{\mathbf{v}}

\newcommand{\PG}{\operatorname{PG}}
\newcommand{\binomq}[2]{\genfrac{[}{]}{0pt}{}{#1}{#2}_q}

\newcommand{\Z}{\mathbb{Z}}

\newcommand{\co}{\mathrm{co}}

\usepackage{amsthm,amssymb}

\newcommand{\ex}{{\rm ex}}

\begin{document}

	\title{On the codegree Turán density of projective geometries
		 }
	\author{Xiaona Fang, Yaojun Chen\thanks{Corresponding author: yaojunc@nju.edu.cn}\\
		{\small School of Mathematics, Nanjing University,} {\small Nanjing, 210093, P.R. China}
}
	\date{}
	\maketitle
	
	\begin{abstract}	
		Let $F$ be a $k$-uniform hypergraph, abbreviated as $k$-graph. The codegree Tur\'an density $\gamma(F)$ is the supremum over all $\gamma \in [0,1)$ such that, for arbitrarily large $n$, there exists an $n$-vertex $F$-free $k$-graph $H$ whose every $(k-1)$-subset of vertices lies in at least $\gamma n$ edges. Let $\PG_m(q)$ be the projective geometry of dimension $m$ over finite field $\mathbb{F}_q$.
		In this paper, we prove that $\gamma(\PG_m(q)) \ge \frac{1}{p}> 0$ for all $m$ and $q$,  where $p$ is the smallest prime divisor of $q+1$. This resolves an open problem proposed by Keevash and Zhao (JCT-B, 2007). Moreover, we determine the exact codegree Tur\'an density of $\PG_4(q)$ when $q$ is an odd prime power.

		
	\end{abstract}
	
	\section{Introduction}\label{sec1}
A $k$-graph is a $k$-uniform hypergraph, that is, each edge consists of $k$ vertices. For a $k$-graph $F$,
the Tur\'an number  $\ex(n,F)$ is defined as the maximum number of edges an $n$-vertex $F$-free $k$-graph can have. Originating from the pioneering works of Mantel and Tur\'an in the early twentieth century, Tur\'an-type problems have long been central to extremal combinatorics. To capture the asymptotic growth of $\ex(n,F)$, one defines the \emph{Tur\'an density} of $F$ as
\[
\pi(F) := \lim_{n\to\infty} \frac{\ex(n,F)}{\binom{n}{k}}.
\]
While Tur\'an densities for ordinary graphs ($k=2$) are well understood, the problem becomes far more complex for hypergraphs with $k\ge 3$.
Despite decades of intensive research, it remains an open challenge for determining the Tur\'an densities of 
the two 3-graphs on four vertices with three and four edges 

A natural and widely studied variant is the \emph{codegree Tur\'an density}, introduced by Mubayi and Zhao~\cite{mubayi}. For a $k$-graph $H$ and a subset $S \subseteq V(H)$, let $d_H(S)$ denote the number of edges containing $S$. The \emph{minimum codegree} of $H$, denoted by $\delta_{\mathrm{co}}(H)$, is the minimum of $d_H(S)$ ranging over all $(k-1)$-subsets $S \subseteq V(H)$. The codegree Tur\'an number $\ex_{\mathrm{co}}(n,F)$ is then defined as the maximum value of $\delta_{\mathrm{co}}(H)$ among all $n$-vertex $F$-free $k$-graphs $H$. The corresponding \emph{codegree Tur\'an density} is given by the limit
\[
\gamma(F) := \lim_{n\to\infty} \frac{\ex_{\mathrm{co}}(n,F)}{n},
\]
which is guaranteed to exist~\cite{mubayi}. And it is not hard to see that $\gamma(F) \le \pi(F)$.
 
 The codegree Turán density is difficult to determine and remains known for only a few  $k$-graphs. Among these, projective geometries constitute a class of highly significant and widely studied hypergraphs.
 Let $\mathbb{F}_q$ denote the finite field with $q$ elements. 
 The projective geometry of dimension $m$ over $\mathbb{F}_q$, denoted by $\PG_m(q)$, is the following $(q+1)$-graph. Its vertex set consists of all one-dimensional subspaces of $\mathbb{F}_q^{m+1}$. Its edges are the two-dimensional subspaces of $\mathbb{F}_q^{m+1}$, in which for each two-dimensional subspace, the set of one-dimensional subspaces that it contains is an edge of the hypergraph $\PG_m(q)$.

 In \cite{mubayi1}, Mubayi showed that $\gamma(\PG_2(2)) = \frac{1}{2}$, where $\PG_2(2)$ is the Fano plane.
 Later, Keevash \cite{keevash1} established the exact codegree extremal number $\ex_{\co}(n, \PG_2(2))$, for which an alternative proof was subsequently provided by DeBiasio and Jiang \cite{debia}.
 In 2007,
 Keevash and Zhao \cite{keevash} studied the codegree Tur\'an density for other projective
geometries $\PG_m(q)$ and obtained the following result.

\begin{theo}(Keevash and Zhao \cite{keevash})\label{thm0}
	The codegree density of projective geometries satisfies
	$
	\gamma(\PG_m(q)) \le 1-\frac{1}{m}.
	$
	Equality holds whenever $m=2$ and $q$ is $2$ or odd, and whenever
	$m=3$ and $q$ is $2$ or $3$.
\end{theo}

They also proposed the following open problem.
\begin{prob}(Keevash and Zhao \cite{keevash})\label{prob1}
	Is $\gamma(\PG_m(q)) > 0$ for all $m$ and $q$?
\end{prob}

Later, Zhang and Ge \cite{zhang} proved that $\gamma(PG_2(q)) = \frac{1}{2}$ for prime power $q \equiv 2 \pmod 3$, and $\gamma(PG_3(q)) = \frac{2}{3}$ for prime power $q \equiv 1 \pmod 2$ or $q \equiv 2 \pmod 3$.
\vskip 0.5em

In this paper, we completely resolve Problem \ref{prob1} and answer it in the affirmative.

\begin{theo}\label{th1}
	Let $q$ be a prime power and $p$ be the smallest prime divisor of $q+1$.  Then
	\[
	\gamma(\PG_m(q)) \ge \frac{1}{p} >0.
	\]
\end{theo}

\begin{remark}
If $q$ is an odd prime power, we have $p=2$. Consequently, the lower bound $\gamma(PG_2(q)) \geq \frac{1}{2}$ established in Theorem \ref{th1} is sharp.
\end{remark}

    Since $4\equiv -1\pmod 5$, we have $2^{4k+2}+1 =4^{2k+1}+1 \equiv (-1)^{2k+1}+1\equiv 0\pmod 5.$ Thus $5\mid 2^{4k+2}+1$. By Theorem \ref{th1}, we have the following corollary.

  \begin{cor}
  Let $k\geq 0$ be an integer. Then
      $$\gamma(\PG_m(2^{4k+2})) \ge \frac{1}{5}.$$
  \end{cor}  

Moreover, for $m=4$ and an odd prime power $q$, we determined the exact codegree Tur\'an density of $\PG_4(q)$.
\begin{theo}\label{th2}
	Let $q$ be an odd prime power, then
	\[
	\gamma(\PG_4(q))= \frac34.
	\]
\end{theo}

\section{Proofs of the main results}

In this section, we give the proofs of the main results.
The following lemma presents the basic properties of $\PG_m(q)$.

\begin{lem}\label{lem}(\cite{hirs})
	The hypergraph $\PG_m(q)$ has the following properties:
	\vskip 0.5em
	(i) $|E(\PG_m(q))|=\frac{(q^{m+1}-1)(q^m-1)}{(q^2-1)(q-1)}; $
	\vskip 0.3em
    
	(ii) each fixed vertex is contained in exactly $\frac{q^m-1}{q-1}$ edges;
	\vskip 0.3em
	
	(iii) any two different vertices  belong to exactly one hyperedge.
\end{lem}

\noindent\textbf{Proof of Theorem \ref{th1}}.
Since $\PG_a(q) \subseteq \PG_b(q)$ for $a \le b$, we always have $\gamma(\PG_a(q)) \le \gamma(\PG_b(q))$. Therefore, to show $\gamma(\PG_m(q))>0$, it suffices to show that $\gamma(\PG_2(q))>0$. 

Let $d>1$ be an integer such that
$d\mid q+1$.  We shall construct, for every sufficiently large $n$, a $\PG_2(q)$-free
$(q+1)$-graph $H_n$ on $n$ vertices with
\[
\delta_{\co}(H_n)\ge \left\lfloor\frac{n}{d}\right\rfloor-q.
\]
This immediately implies $\gamma(\PG_2(q))\ge 1/d$.

Partition the vertex set of $H_n$ into $d$ almost equal parts, i.e., 
$$V(H_n )= V_0\cup V_1\cup\cdots \cup V_{d-1}$$
with $|V_i|=\lfloor n/d\rfloor$ or $\lceil n/d\rceil$.
For a vertex $u$, we write
$\chi(u)=s \in \Z_d$
if and only if
$ u\in V_s$. And $\{u_1,\dots,u_{q+1}\}$ is an edge of $H_n$ if and only if
\[
\chi(u_1)+\cdots+\chi(u_{q+1})\equiv 1\pmod m.
\]

For every $q$-set $S=\{u_1,\dots,u_{q}\}$ in $H_n$, a vertex $y\in N(S)$ if and only if $\chi(y)=j$, where
\(
j\equiv 1-\sum_{i=1}^{q}\chi(u_i)\pmod d.
\)
Thus $ V_j \setminus S \subseteq N(S)$.  Therefore,
$
	\gamma(H_n)
	\ge
	\left\lfloor\frac nd\right\rfloor-q.
$

	Suppose, for a contradiction, that $H_n$ contains a copy of $\PG_2(q)$.  Then for any edge $e\in E(\PG_2(q))$, we have $	\sum_{u\in e}\chi(u)\equiv 1\pmod d$. Since there are $q^2+q+1$ edges in $\PG_2(q)$ by Lemma \ref{lem}, we have
	\begin{equation}\label{eq1}
			\sum_{e\in E(\PG_2(q))}\sum_{u\in e} \chi(u)\equiv  q^2+q+1 \pmod d.
	\end{equation}

 Interchanging the order of summation of the left-hand sides of (\ref{eq1}) gives
	\[
		\sum_{e\in E(\PG_2(q))}\sum_{u\in e} \chi(u)
	=\sum_u \chi(u) \cdot |\{e:u\in e\}|=\sum_u \chi(u)(q+1)\equiv 0\pmod d,
	\]
	since $d\mid(q+1)$. On the other hand, we have $q\equiv -1\pmod d$, and then
$$q^2+q+1\equiv (-1)^2+(-1)+1\equiv 1\pmod d,$$  a contradiction. Therefore, $H_n$ does not contain a copy of $\PG_2(q)$.

Take $d=p$, the conclusion follows.
$\hfill \blacksquare$

\vskip 2mm

\noindent\textbf{Proof of Theorem \ref{th2}}.
By Theorem \ref{thm0}, it suffices to show that $\gamma(\PG_4(q))\ge \frac34$. 
Let $H_n$ be a $(q+1)$-graph on $n$ vertices with $V(H_n )= V_1\cup V_2\cup V_3\cup V_{4}$, where $|V_i|=\lfloor n/4\rfloor$ or $\lceil n/4\rceil$. A $(q+1)$-set $e$ is an edge of $H_n$ if and only if exactly two of the four integers
$|e\cap V_1|$, $|e\cap V_2|$, $|e\cap V_3|$, $|e\cap V_4|$
are odd.
We will prove that 
$
\delta_{\co}(H_n)\ge \frac{3n}{4}-q-1
$ and $H_n$ is $\PG_4(q)$-free.
And then the result follows.

Firstly, we prove that $
\delta_{\co}(H_n)\ge \frac{3n}{4}-q-1$.
Let $S \subseteq V(H_n)$ be an arbitrary $q$-element vertex set. 
For each $i \in [4]$, we denote the part-counts of $S$ by $s_i = |S \cap V_i|$. 
Since $\sum_{i=1}^4 s_i = |S| = q$ is odd, exactly $1$ or $3$ of the integers $s_1, s_2, s_3, s_4$ are odd. 
We proceed by analyzing these two cases.

If exactly one of $s_1,s_2,s_3,s_4$ is odd,
let $j$ be the unique index such that $s_j$ is odd. If $v\in V_j\setminus S$, then $s_j$ changes from odd to even, while all other part-counts were already even. Hence $S\cup\{v\}$ has zero odd part-counts, so $S\cup\{v\}$ is not an edge of $H_n$. If $v\in V_i\setminus S$ with $i\ne j$, then one of the originally even part-counts becomes odd while $s_j$ remains odd. Hence $S\cup\{v\}$ has exactly two odd part-counts, so $S\cup\{v\}$ is an edge of $H_n$. Thus
\(
d_{H_n}(S)
=|V\setminus S|-|V_j\setminus S|
\ge n-q-|V_j| \ge \frac{3n}{4}-q-1.
\)

If exactly three of $s_1,s_2,s_3,s_4$ are odd,
let $j$ be the unique index such that $s_j$ is even. If $v\in V_j\setminus S$, then $S\cup\{v\}$ has four odd part-counts, so it is not an edge of $H_n$.
If $v\in V_i\setminus S$ with $i\ne j$, then one of the three originally odd part-counts becomes even, and the other two odd part-counts remain odd. Hence $S\cup\{v\}$ has exactly two odd part-counts, so it is an edge of $H_n$.
Thus 
\(d_{H_n}(S)
\ge n-q-\left\lceil \frac n4\right\rceil
\ge \frac{3n}{4}-q-1.
\)

Therefore, $\delta_{\co}(H_n)\ge \frac{3n}{4}-q-1$, as desired.

\vskip 0.3em
Now we show that $H_n$ is $\PG_4(q)$-free. Assume, to the contrary, that $H_n$ contains a copy of $\PG_4(q)$. 
Then $H_n$ contains a $\PG_3(q)$. We have the following claim.

\vskip 0.5em
\noindent\textbf{Claim 1.} \label{c1} If $\PG_3(q) \subseteq H_n$, then either each
\(
|C_i|
\)
is even, or each
\(
|C_i|
\) is odd, where $C_i= V_i \cap V(\PG_3(q))$ for $i\in [4]$. 
\begin{proof}
	Let $M$ be the number of edges of $\PG_3(q)$. Then by Lemma \ref{lem}, 
	$$M=\frac{(q^4-1)(q^3-1)}{(q^2-1)(q-1)}=(q^2+1)(q^2+q+1).
	$$
	Since $q$ is odd, $q^2\equiv 1\pmod 4$. Thus $
	q^2+1\equiv 2\pmod 4$ and $q^2+q+1 \equiv 
    1$ or $3 \pmod 4$. Therefore, $M\equiv 2\pmod 4.$
	In particular, $2M\equiv 0\pmod 4.$
	
	For $i\in [4]$, let $c_i=|C_i|$ and $\alpha_i(e)=|e\cap C_i|$ for any edge $e\in E(\PG_3(q))$.
	Let
	\[
	O_i=|\{e\in E(\PG_3(q)):\alpha_i(e)\text{ is odd}\}|.
	\]
	We claim that
	$O_i\equiv c_i(c_i+q^2+q)\pmod 4.$
	To prove this, consider
$\sum_{e\in E(\PG_3(q))} \alpha_i(e)^2.$
	For a fixed edge $e$, $\alpha_i(e)^2$ counts ordered pairs $(x,y)$ such that $x,y\in e\cap C_i.$
	Therefore, the whole sum counts triples $(e,x,y)$, where $e\in E(\PG_3(q))$ and $x,y\in e\cap C_i.$ 
	
	We count these triples in two parts. 
    If $x = y$, then there are $c_i$ ways to choose the vertex $x$. 
	By Lemma \ref{lem}, the number of edges $e$ containing $x$ in $PG_3(q)$ is exactly $q^2+q+1$. 
	Hence, the triples satisfying $x = y$ contribute $(q^2+q+1)c_i.$	
	If $x\ne y$, then there are $c_i(c_i-1)$ ordered pairs $(x,y)$ of distinct vertices of $C_i$. Since any two distinct projective points determine a unique projective line, these triples contribute $
	c_i(c_i-1).$
	Therefore
	\[
	\sum_{e\in E(\PG_3(q))} \alpha_i(e)^2
	=(q^2+q+1)c_i+c_i(c_i-1)
	=c_i(c_i+q^2+q).
	\]

	On the other hand, if $\alpha_i(e)$ is even, then $\alpha_i(e)^2\equiv 0\pmod 4$. If $\alpha_i(e)$ is odd, then $\alpha_i(e)^2\equiv 1\pmod 4$. Hence $
	\sum_{e\in E(\PG_3(q))} \alpha_i(e)^2\equiv O_i\pmod 4.$
	So $
	O_i\equiv c_i(c_i+q^2+q)\pmod 4.$
	And $
	O_1+O_2+O_3+O_4
	\equiv
	\sum_{i=1}^4 c_i(c_i+q^2+q)
	\pmod 4.$ Since every edge contributes exactly $2$ to $O_1+O_2+O_3+O_4$, we have $
	O_1+O_2+O_3+O_4=2M.$
	Because $2M\equiv 0\pmod 4$, we get 
	\begin{equation}\label{eq2}
		\sum_{i=1}^4 c_i(c_i+q^2+q)\equiv 0\pmod 4.
	\end{equation}
	
	Let $P=|\{i\in[4]:c_i\text{ is odd}\}|.$ We will show that $P$ is either $0$ or $4$. 
    
    If $q\equiv 1\pmod 4$, then $
	q^2+q\equiv 1+1\equiv 2\pmod 4.$
	If $c_i$ is even, then $c_i\equiv 0$ or $2\pmod 4$, and $c_i(c_i+q^2+q)\equiv 0\pmod 4$. 
	If $c_i$ is odd, then $c_i\equiv 1$ or $3\pmod 4$, and $
	c_i(c_i+q^2+q)\equiv 3\pmod 4.$	
	Therefore, $
	\sum_{i=1}^4 c_i(c_i+q^2+q)\equiv 3P\pmod 4.$
	By (\ref{eq2}), we have $
	3P\equiv 0\pmod 4.$
	Because $3$ is invertible modulo $4$, this implies $
	P\equiv 0\pmod 4.$
	Since $P\in\{0,1,2,3,4\}$, we obtain $
	P\in\{0,4\}.$
	
	If $q\equiv 3\pmod 4$, then $
	q^2+q\equiv 1+3\equiv 0\pmod 4.$
	Therefore $c_i(c_i+q^2+q)\equiv c_i^2\pmod 4.$
	Note that $c_i^2\equiv 1 \pmod 4$ if $c_i$ is odd and $c_i^2\equiv 0 \pmod 4$ if $c_i$ is even. Thus, $\sum_{i=1}^4 c_i(c_i+q^2+q)\equiv P\pmod 4.$ By (\ref{eq2}), we have $
	P\equiv 0\pmod 4$, that is, $P\in\{0,4\}.$ 
	
	Combining the arguments above, we can see that the assertion holds.
\end{proof}

Recall that $H_n$ contains a copy of $\PG_4(q)$ and $V(H_n )= V_1\cup V_2\cup V_3\cup V_{4}$.
In the vector space $\F_2^4$, let
$
\vv_1=(1,0,0,0), \vv_2=(0,1,0,0), \vv_3=(0,0,1,0), \vv_4=(0,0,0,1)
$
be the standard basis vectors, and let $\0=(0,0,0,0)$ and
$
\one=(1,1,1,1).
$
For each vertex $u\in V(\PG_4(q))$, define $\chi(u)\in\F_2^4$ by
\[
\chi(u)=\vv_i\quad\text{if and only if}\quad u\in V_i.
\]
For a subgraph $F\subseteq \PG_4(q)$, define its parity vector by
\[
f(F)=\sum_{u\in V(F)}\chi(u)\in\F_2^4.
\]
The sum is taken in $\F_2^4$. 
For every projective solid $\varOmega\cong \PG_3(q)$ inside $\PG_4(q)$, we have
$
f(\varOmega)\in\{\0,\one\}
$ by Claim 1.

By the definition of $H_n$,
 there exists no $i\in[4]$ such that $V(\PG_4(q)) \subseteq V_i$.
Choose two vertices
$
x,y\in V(\PG_4(q))
$
with
$
\chi(x)=\vv_a,
\chi(y)=\vv_b,
$
where $a,b\in\{1,2,3,4\}$ and $a\ne b$.
Let $\mathcal S_x$ be the set of projective solids $\varOmega\cong \PG_3(q)$ in $\PG_4(q)$ such that $x\in\varOmega$. We compute
$
\sum_{\varOmega\in\mathcal S_x}f(\varOmega)
$
in two ways in the vector space $\F_2^4$.

On the one hand, since every $f(\varOmega)$ is either $\0$ or $\one$, we have
$
\sum_{\varOmega\in\mathcal S_x}f(\varOmega)\in\{\0,\one\}.
$
On the other hand,
\begin{equation}\label{eqf}
\sum_{\varOmega\in\mathcal S_x}f(\varOmega)
=
\sum_{\varOmega\in\mathcal S_x}\sum_{z\in V(\varOmega)}\chi(z)=
\sum_{z\in V(\PG_4(q))} N_x(z)\chi(z),
\end{equation}
where $N_x(z)$ denotes the number of projective solids $\PG_3(q)$ containing both $x$ and $z$.

We now compute $N_x(z)$ modulo $2$.
First consider $z=x$. Let $X\le \F_q^5$ be the one-dimensional vector subspace corresponding to $x$. A projective solid $\varOmega$ through $x$ corresponds to a four-dimensional vector subspace $W\le \F_q^5$ with $X\subseteq W$. Equivalently, it corresponds to a three-dimensional subspace of the quotient vector space $\F_q^5/X$, which has dimension $4$. The number of such subspaces is
$
\binomq{4}{3}=\binomq{4}{1}=\frac{q^4-1}{q-1}=q^3+q^2+q+1,
$
where $\binomq{n}{k}$ is the Gaussian binomial coefficients.
Since $q$ is odd, the four terms $q^3,q^2,q,1$ are all odd. Therefore,
$
N_x(x)\equiv 0\pmod 2.
$

Now take $z\ne x$. Let $U\le\F_q^5$ be the two-dimensional vector subspace corresponds to the unique projective line determined by $x$ and $z$. Then the projective solids containing both $x$ and $z$ correspond to four-dimensional vector subspaces $W\le\F_q^5$ with $U\subseteq W$. Equivalently, they correspond to two-dimensional subspaces of the quotient vector space $\F_q^5/U$, which has dimension $3$. Thus the number of such solids is
$
\binomq{3}{2}=\binomq{3}{1}=\frac{q^3-1}{q-1}=q^2+q+1.
$
Hence
$
N_x(z)\equiv 1\pmod 2.
$

Combining (\ref{eqf}) with the above arguments, we have
\[
\sum_{\varOmega\in\mathcal S_x}f(\varOmega)
=
\sum_{z\ne x}\chi(z)
\]
in $\F_2^4$.
Let
$
\textbf{u} =f(\PG_4(q))=\sum_{z\in V(\PG_4(q))}\chi(z).
$
Since addition and subtraction are the same in $\F_2^4$,
$
\sum_{z\ne x}\chi(z)=\textbf{u} -\chi(x)=\textbf{u} +\chi(x)=\textbf{u} +\vv_a.
$
By $\sum_{\varOmega\in\mathcal S_x}f(\varOmega)\in\{\0,\one\} $, we get
$
\textbf{u} +\vv_a\in\{\0,\one\}.
$
Therefore
$
 \textbf{u} \in\{\vv_a,\one+\vv_a\}.
$
The same argument applied to the vertex $y$ gives
$
\textbf{u}  \in\{\vv_b,\one+\vv_b\}.
$
But the two sets
$
\{\vv_a,\one+\vv_a \}$ and
$
\{\vv_b,\one+\vv_b\}
$
are disjoint because $a\ne b$.
It is a contradiction. 

Therefore, $H_n$ is $\PG_4(q)$-free.
$\hfill \blacksquare$

\section{Concluding Remark}

Given any $k$-graph $F$ and a positive integer $t$, the $t$-blowup $F(t)$ is defined as follows. For each vertex $x \in V(F)$, there correspond vertices $x^1, \dots, x^t$ in $V(F(t))$. For each edge $x_1 \cdots x_k \in E(F)$, the edge set of $F(t)$ contains all $t^k$ edges of the form $x_1^{i_1} \cdots x_k^{i_k}$ with $1 \le i_1, \dots, i_k \le t$. 
Let $F^+$ denote the $k$-graph obtained from $F(k-1)$ by adding a new vertex $y$ and all edges $y x^1 \cdots x^{k-1}$ for each $x \in V(F)$.

In the proof of the upper bound of $\gamma(\PG_m(q))$, Keevash and Zhao \cite{keevash} showed that 
$
\gamma(F^+) \le \frac{1}{2-\gamma(F)}
$
and $\PG_m(q) \subseteq \PG_{m-1}(q)^+$. 
Furthermore, Keevash \cite{keevash1} proved that there is a constant \(\varepsilon>0\) such that $\gamma(\PG_2(4)) < \frac{1}{2} - \varepsilon$. 
By combining these results, we can show that $\gamma(\PG_m(4)) < 1 - \frac{1}{m}$ holds for every $m \geq 2$. 
Indeed, using induction on $m$ with the base case $m=2$, the induction step yields
\[
\gamma(\PG_m(4)) \le \gamma(\PG_{m-1}(4)^+) \le \frac{1}{2-\gamma(\PG_{m-1}(4))} < \frac{1}{2 - (1 - \frac{1}{m-1})} = 1 - \frac{1}{m},
\]
where the strict inequality follows directly from the induction hypothesis. 

\vskip2mm
Let $k\geq 1$ be an integer and $p_k$ the smallest prime such that $p_k \mid 2^{4k}+1$. 
The known results for $\gamma(PG_m(q))$ are summarized in Table~\ref{tab}, where the entries highlighted in red represent our new contributions.
\vskip 2mm
\begin{table}[h]
	\centering
	\renewcommand{\arraystretch}{1.4}
	\begin{tabular}{|c|c|c|c|c|c|c|}
		\hline
		\diagbox{$m$}{$q$}
        & Odd $q$
		& $2$
		& $4$
		& $2^{2k+1}$
		& $2^{4k+2}$
		& $2^{4k}$
        \\
		\hline
		$2$
        & $\frac12$
		& $\frac12$
		& $[\frac13,\frac12)$
		& $\frac12$
		& $[{\color{red}  \frac15 },\frac12]$ 
		& $[ {\color{red} \frac{1}{p_k}},\frac12]$  \\
		\hline
		$3$
        & $\frac23$
		& $\frac23$
		& $[\frac12,\frac23 {\color{red} )}$	
		& $\frac23$
		& $[{\color{red} \frac15},\frac23]$
		& $[{\color{red} \frac{1}{p_k}},\frac23]$ \\
		\hline
		$4$
        & {\color{red} $\frac34$}
		& $[\frac23,\frac34]$
		& $[\frac12,\frac34{\color{red} )}$
		& $[\frac23,\frac34]$
		& $[{\color{red} \frac15},\frac34]$
		&  $[{\color{red} \frac{1}{p_k}},\frac34]$ \\
		\hline
		$m\ge5$
        &  $[ {\color{red}\frac34},\,1-\frac1m]$
		& $[\frac34,\,1-\frac1m]$
		& $[\frac12,\,1-\frac1m {\color{red} )}$
		& $[\frac23,\,1-\frac1m]$
		& $[{\color{red} \frac15},\,1-\frac1m]$ 
		& $[{\color{red} \frac{1}{p_k}},1-\frac1m]$ \\
		\hline
	\end{tabular}
	\caption{Known results of $\gamma(\PG_m(q))$.}
	\label{tab}
\end{table}

\vskip 5mm
\section*{Acknowledgments}
The authors thank Ruilin Zheng for helpful discussions. This research was supported by National Key R\&D Program of China under grant number 2024YFA1013900 and NSFC under grant number 12471327.

\vskip 5mm
\noindent\textbf{Data availability statement} \  Data sharing not applicable to this article as no datasets were generated or analyzed during the current study.

\vskip 5mm
\noindent\textbf{Declarations of conflict of interest}\  The authors declare that they have no known competing financial interests or personal relationships that could have appeared to influence the work reported in this paper.

\end{document}